\documentclass[a4paper,11pt]{article}

\usepackage{cmap}

\usepackage[usenames,dvipsnames]{xcolor}
\definecolor{myblue}{rgb}{0.21, 0.34, 0.74}
\definecolor{mygrey}{rgb}{0.55, 0.57, 0.67}
\definecolor{myred}{rgb}{0.79, 0.0, 0.09}

\usepackage[pdftex, colorlinks=true, bookmarksopen=true, linkcolor=blue, citecolor=blue]{hyperref}

\usepackage[T1]{fontenc}
\usepackage[utf8]{inputenc}

\usepackage{lmodern}
\usepackage{libertine}

\usepackage{enumitem}
\usepackage{bbm}
\usepackage{mathrsfs}
\usepackage{amssymb}
\usepackage{amsmath}
\usepackage{upgreek}
\usepackage{mathtools}

\usepackage{amsthm}
\usepackage[nottoc,notlot,notlof]{tocbibind}

\usepackage[noabbrev,capitalize,nameinlink]{cleveref}

\usepackage{comment}
\usepackage[font=small,labelfont=bf]{caption}

\theoremstyle{plain}
\newtheorem{theorem}{Theorem}[section]

\newtheorem{proposition}[theorem]{Proposition}
\newtheorem{lemma}[theorem]{Lemma}

\newtheorem{remark}[theorem]{Remark}

\usepackage{authblk}

\usepackage{subcaption}

\usepackage[svgnames,pdf]{pstricks}

\usepackage{wrapfig}
\usepackage{upgreek}
\usepackage{bm}
\usepackage[scr=boondoxo]{mathalfa}
\usepackage{comment}
\usepackage[font=small,labelfont=bf]{caption}

\mathchardef\mhyphen="2D 

\renewcommand{\leq}{\leqslant}
\renewcommand{\geq}{\geqslant}

\numberwithin{equation}{section}

\title{ Constructive approximate transport maps with normalizing flows  }

\author{Antonio Álvarez-L\'opez}
\affil{Universidad Aut\'onoma de Madrid}
\author{Borjan Geshkovski}
\affil{Inria \& Sorbonne Université}
\author{Domènec Ruiz-Balet}
\affil{Université Paris Dauphine}

\date{ \today }

\begin{document}
\setlist[itemize,enumerate]{left=0pt}

%
%

\maketitle

%
%

\begin{abstract}

We study an approximate controllability problem for the continuity equation and its application to constructing transport maps with normalizing flows. 
Specifically, we construct time-dependent controls $\theta=(w, a, b)$ in the vector field $x\mapsto w(a^\top x + b)_+$ to approximately transport a known base density $\rho_{\mathrm{B}}$ to a target density $\rho_*$. The approximation error is measured in relative entropy, and $\theta$ are constructed piecewise constant, with bounds on the number of switches being provided. Our main result relies on an assumption on the relative tail decay of $\rho_*$ and $\rho_{\mathrm{B}}$, and provides hints on characterizing the reachable space of the continuity equation in relative entropy. 
	
\bigskip

\noindent \textbf{Keywords.}\quad Normalizing flows, approximate controllability, optimal transport, Pinsker inequality.

\medskip

\noindent \textbf{\textsc{ams} classification.}\quad \textsc{35Q49, 93C15, 68T07, 93C20, 46N10, 49Q22}.
\end{abstract}
	
\thispagestyle{empty}

\setcounter{tocdepth}{2}
\tableofcontents

%
%

\section{Introduction}

By viewing discrete layers as a continuous time variable, the framework of \emph{neural ODEs}---proposed by~\cite{weinan2017proposal} and made feasible in computing practice by~\cite{he2016deep,chen2018neural}---has significantly influenced the study and application of neural networks.
They are transparent objects: each (of many) data-point $x\in\mathbb{R}^d$ is propagated through the flow of the Cauchy problem
\begin{equation} \label{eq: neural.ode.drb}
\begin{cases}
    \dot{x}(t) = v(x(t), \theta(t)) &\text{ for } t\in[0, T]\\
    x(0) = x;
\end{cases}
\end{equation}
the velocity field is, canonically, a perceptron 
\begin{equation} \label{eq: two.layer.mlp}
 v(x,\theta) \coloneqq w\left(a^\top x+b\right)_+,
\end{equation}
with $\theta=(w,a,b)\in\mathbb{R}^{d}\times\mathbb{R}^d\times\mathbb{R}\cong\mathbb{R}^{2d+1}$ being the parameters of the network. (Natural generalizations of \eqref{eq: two.layer.mlp} can be devised, by replacing the inner product by a matrix multiplication or discrete convolution for example.)
We will consistently assume that $\theta(\cdot)\in L^\infty(0, T;\mathbb{R}^{2d+1})$, ensuring the applicability of the Cauchy-Lipschitz theorem.
A neural network can thus be interpreted as the flow map $\Phi^T_\theta:\mathbb{R}^d\to\mathbb{R}^d$ of \eqref{eq: neural.ode.drb} at time $t=T$, yielding a representation $x(T)$ of each data point $x$. This representation is then fed into a last layer that corresponds to a classical machine learning task such as linear or logistic regression, in order to predict the corresponding label. 

Besides serving as a useful abstraction of discrete neural networks, this formalism has had tremendous impact on density estimation through the guise of \emph{normalizing flows} \cite{grathwohl2018ffjord}. This is a popular methodology in machine learning as evidenced by \cite{papamakarios2021normalizing, kobyzev2020normalizing, lipman2022flow, albergo2023stochastic, gabrie2022adaptive, pooladian2023multisample} and the references therein, and goes back at least as \cite{tabak2010density, tabak2013family, dinh2014nice, tomczak2016improving}.
To estimate an unknown probability density $\rho_*$, of which samples are known, one takes a simple initial probability density $\rho_{\mathrm{B}}$, and matches the solution of the continuity (or conservative transport) equation
\begin{equation} \label{eq: cont.eq}
\begin{cases}
\partial_t \rho(t, x) + \mathrm{div}\Big(v(x,\theta(t))\rho(t,x)\Big)=0 &\text{ in } [0, T]\times\mathbb{R}^d\\
\rho(0, x) = \rho_{\mathrm{B}}(x) &\text{ on } \mathbb{R}^d
\end{cases}
\end{equation}
to the known samples. Specifically, given $n$ iid samples $x_1,\ldots, x_n$ of the unknown density $\rho_{*}$, one solves 
\begin{equation} \label{eq: empirical.mle}
    \max_{\theta}\frac{1}{n}\sum_{i=1}^n \log \rho_{\mathrm{B}}(\Phi^{-T}_{\theta}(x_i)) + \log|\mathrm{det}\nabla\Phi^{-T}_{\theta}(x_i)|. 
\end{equation}
(We shall keep this presentation formal, and do not detail what constraints are imposed on $\theta$ to ensure existence of maximizers.)
This is the problem of \emph{maximum (log-)likelihood estimation} on the data. Here $\Phi^{-T}_\theta$ designates the end-time flow map of \eqref{eq: neural.ode.drb} ran backwards in time, from $T$ to $0$; the eponym "normalizing" then stems from the fact that  $\rho_{\mathrm{B}}$ is most often a Gaussian probability density in practice. 

If $\rho_*$ were to be known, then \eqref{eq: empirical.mle} is the empirical version of
\begin{equation} \label{eq: pop.mle}
    \max_{\theta} \int\rho_*(x)\left(\log \rho_{\mathrm{B}}(\Phi^{-T}_\theta(x))+\log|\mathrm{det}\nabla\Phi_\theta^{-T}(x)|\right)\, \mathrm{d} x.
\end{equation}
By the change of variable theorem, we have $\rho_{\mathrm{B}}(\Phi^{-T}_\theta(x))|\mathrm{det}\nabla\Phi_\theta^{-T}(x)|=\rho(T,x)$, and thus 
\begin{align}
    \eqref{eq: pop.mle} &= \max_\theta \int \rho_*(x) \log\rho(T,x)\, \mathrm{d} x \nonumber \\
    & = \min_\theta \int \rho_*(x)\log \frac{\rho_*(x)}{\rho(T,x)}\, \mathrm{d} x - \int \rho_*(x)\log \rho_*(x)\, \mathrm{d} x. \label{eq: min.kl}
\end{align}
This is the question of minimizing the \emph{relative entropy} between $\rho_*$ and $\rho(T)$, or equivalently, the \emph{Kullback-Leibler divergence} between the measures $\mu_*$ and $\Phi^T_{\theta\#}\mu_{\mathrm{B}}$, 
henceforth denoted as
\begin{equation*}
\mathsf{KL}\left(\mu_*\,\|\,\Phi^T_{\theta\#}\mu_{\mathrm{B}}\right):=\int \rho_*(x)\log \frac{\rho_*(x)}{\rho(T,x)}\, \mathrm{d} x,
\end{equation*}
where $\rho_*=\frac{\, \mathrm{d}\mu_*}{\, \mathrm{d} x}$. It is a problem of general interest in Bayesian inference and sampling---all we do here is to parametrize it by densities given as solutions to \eqref{eq: cont.eq}. 

Our focus in this paper is the feasibility problem associated to \eqref{eq: min.kl}---namely, given an arbitrarily small $\varepsilon>0$, we wish to find $\theta$ (depending on $\varepsilon$) such that 
\begin{equation*}
    \int \rho_*(x)\log \frac{\rho_*(x)}{\rho(T,x)}\, \mathrm{d} x\leq \varepsilon.
\end{equation*}
This constitutes an \emph{approximate controllability} or \emph{approximate transportation} problem, which is \emph{nonlinear} in the controls $\theta$. Results for \eqref{eq: cont.eq} are only known in weaker\footnote{see \eqref{eq: pinsker} and \Cref{rem: metrics}.} topologies---such as $L^1$, $L^2$ or Wasserstein distances \cite{ruiz2023neural, ruiz2024control, marzouk2023distribution, alvarez2024interplay, geshkovski2024measure}---and, again in weaker topologies, also for continuity equations with vector fields that are linear in the controls $\theta$ \cite{alabau2012notes, raginsky2024some, duprez2019approximate, scagliotti2023normalizing}.
In addition to addressing the strong topology imposed by the relative entropy, we will require $\theta$ to be piecewise constant, with the aim of counting the number of discontinuities (switches). This is of interest, as these switches can be interpreted as analogous to discrete layers.
In the setting of such $\theta$, given an initial density $\rho_{\mathrm{B}}\in\mathscr{C}^0(\mathbb{R}^d)$, \eqref{eq: cont.eq} has a unique solution $\rho\in\mathscr{C}^0([0,T];L^1\cap L^\infty(\mathbb{R}^d))$ which is explicit and found by the method of characteristics (\Cref{lem: transport.multid.sol}).

\subsection{Main result} 

Our main result is as follows.

\begin{theorem} \label{thm:KLcontrolmultid}
Suppose $\mu_{\mathrm{B}}=\mathcal{N}(m_{\mathrm{B}},\Sigma_{\mathrm{B}})$ for an arbitrary mean $m_{\mathrm{B}}\in\mathbb{R}^d$ and covariance matrix $\Sigma_{\mathrm{B}}\in \mathscr{S}_{d}^+(\mathbb{R}).$  Let  $\rho_*$ be any probability density on $\mathbb{R}^d$ 
and for some $M>0$ and $\sigma_\bullet>0$,
\begin{equation} \label{eq:condlowgaussmultid}
    \rho_*(x)\leq\frac{1}{(2\pi\sigma_\bullet^2)^{\frac{d}{2}}}e^{-\frac{\|x\|^2}{2\sigma_\bullet^2}}\eqqcolon\rho_\bullet(x) \hspace{1cm} \text{ for all }\|x\|\geq M.
    \end{equation}
Then for any $T>0$ and $\varepsilon>0$, there exists $\theta=(w,a,b):[0, T]\to\mathbb{R}^{2d+1}$, piecewise constant with finitely many switches, such that the corresponding solution $\rho$ to \eqref{eq: cont.eq} with data $\rho_{\mathrm{B}}=\frac{\, \mathrm{d}\mu_{\mathrm{B}}}{\, \mathrm{d} x}$ satisfies 
\begin{equation*} 
    \int \rho_*(x)\log\frac{\rho_*(x)}{\rho(T,x)}\, \mathrm{d} x\leq\varepsilon.
\end{equation*}
Equivalently,
\begin{equation*} 
\mathsf{KL}\left(\mu_*\,\|\,\Phi^T_{\theta\#}\mu_{\mathrm{B}}\right)\leq\varepsilon,
\end{equation*}
where $\rho_*=\frac{\, \mathrm{d}\mu_*}{\, \mathrm{d} x}$ and $\Phi^T_\theta:\mathbb{R}^d\to\mathbb{R}^d$ is the flow map of \eqref{eq: neural.ode.drb} at time $t=T$. 
\end{theorem}

We defer a detailed discussion on the setup and extensions of the above result to {\bf\Cref{sec: discussion}}, and instead briefly motivate the proof. Our starting point is \cite{ruiz2024control}, in which $L^1$-approximate controllability is shown:
\begin{equation*}
    \|\rho(T)-\rho_*\|_{L^1(\mathbb{R}^d)}=\left|\Phi^T_{\theta\#}\mu_{\mathrm{B}}-\mu_*\right|_{\mathsf{TV}}\leq \varepsilon.
\end{equation*}
The proof is entirely constructive and hinges on performing exact transportation between piecewise constant approximations of the initial and target densities, along with a continuity argument. This is achieved by moving and deforming each segment of the approximated initial density to match the approximated target density using a piecewise constant control, $\theta$. Consequently, the number of switches in $\theta$ corresponds to the number of segments in the approximation.
In higher dimensions, the proof extends by employing piecewise constant approximations over grids, which result in an exponential number of cells (and thus segments) in the dimension $d$. The argument proceeds via induction on the dimension. Since the number of cells grows exponentially with $d$, the number of switches in $\theta$ similarly scales exponentially with $d$.

One way to make use of \cite{ruiz2024control} is to employ a functional inequality linking  $\mathsf{KL}$ and $\mathsf{TV}$. A well-known example is the (Csiz\'ar-Kullback-)\emph{Pinsker} inequality 
\begin{equation} \label{eq: pinsker}
    |\mu_2-\mu_1|_{\mathsf{TV}}\leq\sqrt{2\cdot\mathsf{KL}(\mu_2\,\|\,\mu_1)}
\end{equation}
which holds for arbitrary probability measures $\mu_1, \mu_2$.
While the inequality sign in \eqref{eq: pinsker} is not in the desired direction so that we can straightforwardly apply it to conclude, it turns out that it can be reversed under stronger conditions on the tails of the measures (see \cite{verdu2014total,Sason2015OnRP,sason2016f}):

\begin{lemma}[Reverse Pinsker inequality] \label{lem:revpinsker1}
Let $\mu_1,\mu_2$ be probability measures on $\mathbb{R}^d$ such that $\mu_2\ll\mu_1$. 
Then
\begin{equation} \label{eq:reversepinsker1}
\mathsf{KL}(\mu_2\,\|\,\mu_1)\leq \frac12\cdot\frac{\log\left\|\frac{\, \mathrm{d}\mu_2}{\, \mathrm{d}\mu_1}\right\|_{L^\infty(\mathbb{R}^d)}}{1-\frac{1}{\left\|\frac{\, \mathrm{d}\mu_2}{\, \mathrm{d}\mu_1}\right\|_{L^\infty(\mathbb{R}^d)}}}\cdot|\mu_2-\mu_1|_{\mathsf{TV}}.
\end{equation}
\end{lemma}

Recall that the Radon-Nikodym theorem only ensures that $\frac{\, \mathrm{d}\mu_2}{\, \mathrm{d}\mu_1}$ is measurable.
Yet for \eqref{eq:reversepinsker1} to be non-void we need 
$$
\left\|\frac{\, \mathrm{d}\mu_2}{\, \mathrm{d}\mu_1}\right\|_{L^\infty(\mathbb{R}^d)}<+\infty.$$ 
(Note that this quantity is always $\geq1$ by definition.)
Suppose instead that $\mu_1$ and $\mu_2$ are absolutely continuous with respect to the Lebesgue measure. We are asking for 
$$
\left\|\frac{\, \mathrm{d}\mu_2}{\, \mathrm{d}\mu_1}\right\|_{L^\infty(\mathbb{R}^d)}=\mathop{\mathrm{esssup}}_{x\in\mathbb{R}^d}\frac{\rho_2(x)}{\rho_1(x)}<+\infty.$$ 
Therefore, to convert the $\mathsf{TV}$ approximation result of \cite{ruiz2024control} to a $\mathsf{KL}$ one, we also ought to ensure that
\begin{equation}\label{eq:condtails}
\mathsf{L}(\rho_*,\rho(T))\coloneqq\lim_{\|x\|\to\infty}\frac{\rho_*(x)}{\rho(T,x)}<+\infty,
\end{equation}
as well as $\rho(T,x)>0$  for all $x\in\mathbb{R}^d$.
Ensuring this condition is in essence the main novelty of our study. The complete proof can be found in {\bf \Cref{sec: proof.thm1}}.

At this point, we can also observe that $\mathsf{KL}$ is not symmetric with respect to its inputs, whence \Cref{thm:KLcontrolmultid} does not directly entail a result when the densities are swapped. Interestingly, minimizing the reverse-$\mathsf{KL}$ is a problem of interest in its own right and is referred to as \emph{variational inference} (see  \cite{wainwright2008graphical, blei2017variational, rezende2015variational, jiang2023algorithms} and the references therein for works on the topic). 
The following result turns out to be a simple modification of the proof of \Cref{thm:KLcontrolmultid}.

\begin{theorem} \label{thm:revKLcontrolmultid}
Suppose $\mu_{\mathrm{B}}=\mathcal{N}(m_{\mathrm{B}},\Sigma_{\mathrm{B}})$ for an arbitrary mean $m_{\mathrm{B}}\in\mathbb{R}^d$ and covariance matrix $\Sigma_{\mathrm{B}}\in \mathscr{S}_{d}^+(\mathbb{R}).$  Let  $\rho_*$ be any probability density on $\mathbb{R}^d$ such that $\rho_*>0$, $\rho_*\log\rho_*\in L^1(\mathbb{R}^d)$ and for some $M>0$ and $\sigma_\bullet>0$,
\begin{equation}\label{eq:condupgaussmultid}
    \rho_*(x)\geq\frac{1}{(2\pi\sigma_\bullet^2)^{\frac{d}{2}}}e^{-\frac{\|x\|^2}{2\sigma_\bullet^2}}\eqqcolon\rho_\bullet(x)\hspace{1cm} \text{ for all }\|x\|\geq M.
    \end{equation}
Then for any $T>0$ and $\varepsilon>0$, there exists $\theta=(w,a,b):[0, T]\to\mathbb{R}^{2d+1}$, piecewise constant with finitely many switches, such that the corresponding solution $\rho$ to \eqref{eq: cont.eq} satisfies 
\begin{equation*}
\int\rho(T,x)\log\frac{\rho(T,x)}{\rho_*(x)}\, \mathrm{d} x\leq\varepsilon.
\end{equation*}
Equivalently
\begin{equation*}
\mathsf{KL}\left(\Phi^T_{\theta\#}\mu_{\mathrm{B}}\,\|\,\mu_*\right)\leq\varepsilon,
\end{equation*}
where $\rho_*=\frac{\, \mathrm{d}\mu_*}{\, \mathrm{d} x}$ and $\Phi^T_\theta:\mathbb{R}^d\to\mathbb{R}^d$ is the flow map of \eqref{eq: neural.ode.drb} at time $t=T$. 
\end{theorem}

We provide the proof in {\bf \Cref{sec: proof.thm2}}.

\subsection{Discussion} \label{sec: discussion}

We turn to commenting the various assumptions made in \Cref{thm:KLcontrolmultid,thm:revKLcontrolmultid}. 
In {\bf\Cref{sec: generalities}}, we discuss estimates on the number of switches (\Cref{rem: switches}), the assumption of a Gaussian initial density (\Cref{rem:extKL}) and the possible removal of \eqref{eq:condlowgaussmultid} (\Cref{prop: xlogx}), considering different metrics (\Cref{rem: metrics}), and estimation from finitely many samples (\Cref{rem: statistics}). 
In {\bf\Cref{sec: linear.model}}, we discuss changing the model \eqref{eq: two.layer.mlp} to the one considered in \cite{scagliotti2023normalizing}, which is \emph{linear} in $\theta$.

\subsubsection{Generalities} \label{sec: generalities}

\begin{remark}[Number of switches] \label{rem: switches}
    The controls $\theta$ from \Cref{thm:KLcontrolmultid} have at most 
\begin{equation*}
\left\lceil \frac{2R(\varepsilon)}{h(\varepsilon)}\right\rceil^d(d+10)+2d
\end{equation*} 
switches, where $R(\varepsilon),h(\varepsilon)>0$ are such that
\begin{equation*}
    \int_{[-R(\varepsilon),R(\varepsilon)]^d}\rho_{\mathrm{B}}(x)\, \mathrm{d} x \wedge \int_{[-R(\varepsilon),R(\varepsilon)]^d}\rho_*(x)\, \mathrm{d} x >1-\varepsilon
\end{equation*}
and 
\begin{equation*}
    \int_{[-R(\varepsilon),R(\varepsilon)]^d}\left|\rho_{\mathrm{B}}^{h_\varepsilon}(x)-\rho_{\mathrm{B}}(x)\right|\, \mathrm{d} x\vee\int_{[-R(\varepsilon),R(\varepsilon)]^d}\left|\rho^{h_\varepsilon}_*(x)-\rho_*(x)\right|\, \mathrm{d} x<\varepsilon,
\end{equation*}
where $\rho_{\mathrm{B}}^{h_\varepsilon}$ and $\rho_*^{h_\varepsilon}$ are piecewise constant interpolants of $\rho_{\mathrm{B}}$ and $\rho_*$ on a regular grid of the hypercube $[-R(\varepsilon),R(\varepsilon)]^d$ with cells of spacing $h(\varepsilon)$. The exponential growth with $d$ of the number of switches is due to \cite[Theorem 1]{ruiz2024control}, where the error scales inversely with the cell-size 
of a multidimensional grid.

In \Cref{thm:revKLcontrolmultid}, $\theta$ has at most 
\begin{equation*}
\left\lceil \frac{2R(\varepsilon)}{h(\varepsilon)}\right\rceil^d(d+10)+2d
\end{equation*} 
switches, where $R(\varepsilon)>0$ and $h(\varepsilon)>0$ are as in \Cref{thm:KLcontrolmultid}.
\end{remark}

\begin{remark}[Beyond Gaussians] \label{rem:extKL}
\Cref{thm:KLcontrolmultid,thm:revKLcontrolmultid} can be generalized beyond the case of Gaussian initial densities. For \Cref{thm:KLcontrolmultid}, it suffices to replace \eqref{eq:condlowgaussmultid} by the following three conditions:
\begin{itemize}
\item The densities $\rho_{\mathrm{B}}$ and $\rho_*$ have full support\footnote{If $t\mapsto\rho_{*}(ut)$ is compactly supported for some $u\in\mathbb{S}^{d-1}$, then  $\lim_{|t|\to\infty}\rho_{*}(tu)/\rho(T,tu)=0$ trivially. This means we only have to care about the directions in which the support is unbounded.}. Thus, $\rho_{\mathrm{B}}(x) = e^{-U_{\mathrm{B}}(x)}$ as well as $ \rho_*(x) = e^{-U_*(x)}$.
 \item There exists $M_*>0$ such that
  \begin{equation}\label{eq:condlowmultid*}
  U_*(x)\gtrsim\|x\|\hspace{1cm}\text{for all }\|x\|>M_*.   
 \end{equation}
 \item Given $A\in\mathscr{M}_{d\times d}(\mathbb{R})$ with $\operatorname{spec}(A)\subset(0,+\infty)$, there exists $\lambda_A>0$ such that the lower convex envelope, defined for all $x\in\mathbb{R}^d$ by 
 \begin{equation*}
     \operatorname{conv}U_*(x)=\sup_{\substack{U\leq U_*\\U \mathrm{convex}}} U(x)
 \end{equation*}
 satisfies 
\begin{equation}\label{eq:condlowmultidB}\lim_{\|x\|\to \infty}\frac{\operatorname{conv}U_*(\lambda_A x)}{U_{\mathrm{B}}(A x)}=+\infty.\end{equation}
\end{itemize}
Similarly, to extend \Cref{thm:revKLcontrolmultid} we replace \eqref{eq:condupgaussmultid} with these conditions and instead consider the upper convex envelope. Furthermore, we invert the quotient in \eqref{eq:condlowmultidB}. 
We provide the arguments of the proofs in {\bf\Cref{sec: on.remark.gauss}}.
\end{remark}

To build on \Cref{rem:extKL}, one can inquire whether \eqref{eq: two.layer.mlp}--\eqref{eq: cont.eq} can be used to couple two densities with tails of different nature. To this end, we suspect that the globally Lipschitz character of \eqref{eq: two.layer.mlp} needs to be violated, in which case we could obtain an affirmative answer as illustrated by the following partial result. 

\begin{proposition} \label{prop: xlogx}
Fix $p>0$. Consider
\begin{equation}\label{eq: log.transport}
\begin{cases}
\partial_t\rho(t,x) +\partial_x((|x|\log x_+)_+\rho(t,x))=0 &\text{ in } \mathbb{R}_{>0}\times \mathbb{R}\\
\rho(0, x)=\rho_{\mathrm{B},p}(x) &\text{ in } \mathbb{R}
\end{cases}
\end{equation}
where $\rho_{\mathrm{B}, p}\in \mathscr{C}^0(\mathbb{R}^d)$ with $\rho_{\mathrm{B}, p}\propto \exp(-|x|^p)$. Then, for every $q\in(0, p]$ there exist a time $T_q>0$ such that the unique weak solution $\rho$ to \eqref{eq: log.transport} satisfies
\begin{equation*}
\lim_{x\to+\infty}\frac{\rho(T_q,x)}{\rho_{\mathrm{B},q}(x)}<+\infty.
\end{equation*}
\end{proposition}

We provide the proof in {\bf\Cref{sec: proof.propxlogx}}.
\Cref{prop: xlogx} asserts that one can avoid assuming \eqref{eq:condlowgaussmultid} in \Cref{thm:KLcontrolmultid} by sacrificing the globally Lipschitz character of the nonlinearity and, in turn, use \eqref{eq: cont.eq} to transport a Gaussian to a Laplace density, for example.

\begin{remark}[Metric] \label{rem: metrics}
We discuss extensions of \Cref{thm:KLcontrolmultid} to other divergences which are common in statistics and information theory.
\begin{itemize}
    \item Consider the squared \emph{Hellinger distance}:
    \begin{equation*}
        \mathsf{H}^2(\mu_1, \mu_2)\coloneqq\int \left(\sqrt{\mu_1(\, \mathrm{d} x)}-\sqrt{\mu_2(\, \mathrm{d} x)}\right)^2.
    \end{equation*}
    One has
    \begin{equation*}
    \mathsf{H}^2(\mu_1,\mu_2)\leq |\mu_1-\mu_2|_{\mathsf{TV}(\mathbb{R}^d)}\leq \sqrt{2}\cdot\mathsf{H}(\mu_1,\mu_2).
    \end{equation*}
    Thus a direct application of \cite[Theorem 1]{ruiz2024control} gives
    \begin{equation*}
\mathsf{H}^2\left(\Phi^T_{\theta\#}\mu_{\mathrm{B}}, \mu_*\right)\leq\varepsilon.
\end{equation*}
    \item More generally one could consider any \emph{$f$-divergence} between $\mu_1\ll\mu_2$. For a convex function $f:(0,+\infty)\to\mathbb{R}$ with $f(1)=0$ and $\lim_{t\to0^+}f(t)\in [0,+\infty]$, the $f$-divergence \cite{Polyanskiy_Wu_2024} is defined as
\begin{equation}\label{eq:fdiv}
   \mathsf{D}_f(\mu_1,\mu_2)= \int f\left(\frac{\, \mathrm{d}\mu_1}{\, \mathrm{d}\mu_2}\right)\, \mathrm{d}\mu_2
\end{equation}
For instance, $f(t)=t\ln t$ corresponds to $\mathsf{KL}$, $f(t)=|t-1|$ corresponds to  $\mathsf{TV}$,  $f(t)=(\sqrt{t}-1)^2$ corresponds to $\mathsf{H}^2$, and $f(t)=|t-1|^2$ corresponds to $\chi^2$.
From \eqref{eq:fdiv} one sees that $f_1\lesssim f_2$ is equivalent to $\mathsf{D}_{f_1}(\mu_1,\mu_2)\lesssim \,\mathsf{D}_{f_2}(\mu_1,\mu_2)$ for all $\mu_1,\mu_2$ with $\mu_1\ll\mu_2$. 
Thus by virtue of \cref{thm:KLcontrolmultid} we deduce that for any $\varepsilon>0$ and $f$ as above,
\begin{equation*}
\text{if}\quad\sup_{t>0}\frac{f(t)}{t\ln t}<+\infty\hspace{1cm}\text{then}\quad\mathsf{D}_f\left(\mu_*\,\|\,\Phi^T_{\theta\#}\mu_{\mathrm{B}}\right)\leq\varepsilon.
\end{equation*}
The condition is satisfied by $f(t)=|t-1|$ and $f(t)=(\sqrt{t}-1)^2$.

\item One could also consider the \emph{Rényi entropy} \cite{Polyanskiy_Wu_2024}: for any $\lambda>0$ with $\lambda\neq1$, the Rényi entropy of order $\lambda$ between   $\mu_1\ll\mu_2$, is defined as
\begin{equation*}
  \mathsf{D}_\lambda (\mu_1,\mu_2)\coloneqq\frac{1}{\lambda-1}\log\int_{\mathbb{R}^d} \left(\frac{\, \mathrm{d}\mu_1}{\, \mathrm{d} \mu_2}\right)^{\lambda}\, \mathrm{d} \mu_2.
\end{equation*}
By continuity $\mathsf{D}_1 =\mathsf{KL}$, and similarly $\mathsf{D}_\infty(\mu_1,\mu_2)=\log \operatorname{ess sup}\frac{\, \mathrm{d} \mu_1}{\, \mathrm{d} \mu_2}$.
The latter is known as \emph{worst-case regret}. 
If $\mu_1$ and $\mu_2$ have continuous densities denoted $\rho_1$ and $\rho_2$ respectively, then $\mathsf{D}_\infty(\mu_1,\mu_2)\leq \varepsilon$ implies 
\begin{equation}\label{eq: uniform}
    \|\rho_1-\rho_2\|_{L^\infty(\mathbb{R}^d)}\lesssim\varepsilon.
\end{equation} 
Our methods do not extend to the case of worst-case regret since we rely on the $L^1(\mathbb{R}^d)$--approximation of densities of \cite{ruiz2024control}, which doesn't imply \eqref{eq: uniform}. 
On the other hand, for any fixed $\mu_1\ll\mu_2$, the map $\lambda\in[0,+\infty]\mapsto \mathsf{D}_\lambda (\mu_1,\mu_2)$ is non-decreasing \cite{van2014renyi}. \cref{thm:KLcontrolmultid} thus implies that
\begin{equation*}
\mathsf{D}_\lambda\left(\mu_*\,\|\,\Phi^T_{\theta\#}\mu_{\mathrm{B}}\right)\leq\varepsilon
\end{equation*}
if $0\leq \lambda\leq 1$.
\end{itemize}
\end{remark}

\begin{remark}[Statistics] \label{rem: statistics}
    In practice one only has access to $n$ samples of the target density $\rho_*$. Much like \cite{ruiz2024control}, we can use \Cref{thm:KLcontrolmultid} to obtain sample complexity bounds regarding the explicit construction of $\theta$. To this end, one first employs a non-parametric estimator $\widehat{\rho}_{*,n}$ of $\rho_*$ using the $n$ samples (e.g., via a kernel density estimator), and then uses \Cref{thm:KLcontrolmultid} to approximate $\widehat{\rho}_{*, n}$ to any desired accuracy $\varepsilon$. In doing so, sample complexity-wise, the specific construction of \Cref{thm:KLcontrolmultid} can do as well as any non-parametric estimator of $\rho_*$, but at the additional cost that is the number of switches, which may be exponential in the dimension $d$.
\end{remark}

\subsubsection{Linear vector field} \label{sec: linear.model}

The related work \cite{scagliotti2023normalizing} considers the simpler vector field which is linear in $\theta$:
\begin{equation} \label{eq: node}
    v(x,\theta) = w\upsigma(x)+b, \hspace{1cm} \theta=(w,b)\in\mathscr{M}_{d\times d}(\mathbb{R})\times\mathbb{R}^d,
\end{equation}
with $\upsigma:\mathbb{R}\to\mathbb{R}$ assumed Lipschitz and smooth, and defined component-wise. 

When $\upsigma(x)=(x)_+$, \eqref{eq: node} cannot yield \eqref{eq:condtails} in general, because $\upsigma|_{\mathbb{R}_{\leq0}}=0$. Indeed for $R>0$ large enough (depending on $b$) and $x\in (\mathbb{R}_{<0})^d\cap \{x\colon\|x\|\geq R\}$, the solution to \eqref{eq: cont.eq} reads
\begin{equation*}
\rho(t,x) = \rho_{\mathrm{B}}\left(x-\int_0^tb(s)\, \mathrm{d} s\right), \hspace{1cm}\text{ for  } x\in (\mathbb{R}_{<0})^d\cap \{x\colon\|x\|\geq R\}.
\end{equation*}
For $x$ away from the origin, $\rho(t,x)$ is simply a translation of the initial data $\rho_B$. Thus, the tails will not be generally be of the same order. 
Much like \eqref{eq: two.layer.mlp} however (see \cite{li2022deep, ruiz2023neural, cheng2023interpolation}), one can use \eqref{eq: neural.ode.drb}--\eqref{eq: node} to transport an empirical measure $\mu_0=\frac{1}{n}\sum_{i\in[n]}\delta_{x^{(i)}}$ to a target $\mu_1=\frac{1}{n}\sum_{i\in[n]}\delta_{y^{(i)}}$. In the case where $\upsigma$ is smooth, this can be done by using techniques from geometric control theory based on computing iterated Lie brackets, in the spirit of the celebrated Chow-Rashevskii theorem \cite{agrachev2022control, cuchiero2020deep, scagliotti2021deep, tabuada2022universal, elamvazhuthi2022neural}. The result is not known for \eqref{eq: node} with $\upsigma(x)=(x)_+$. 

We consider
\begin{equation*}
\left(x^{(i)},y^{(i)}\right)\in\mathbb{R}^d\times\mathbb{R}^d, \hspace{1cm} \text{ for } i\in[n],
\end{equation*}
and, to avoid unnecessary technical details, work under the benign  assumption that $x^{(i)}\neq x^{(j)}$ and $y^{(i)}\neq y^{(j)}$ for $i\neq j$. The following holds.

\begin{proposition} \label{thm:exactcontrol}
Let $T>0$ and suppose $\upsigma(x)\equiv(x)_+$ in \eqref{eq: node}.
Then there exists $\theta=(w,b):[0,T]\to \mathscr{M}_{d\times d}(\mathbb{R})\times\mathbb{R}^d$, piecewise constant with at most $4n+2$ switches, such that the flow map $\Phi^t_\theta:\mathbb{R}^d\to\mathbb{R}^d$ ($t\in[0,T]$) of \eqref{eq: node}
satisfies 
\begin{equation*}
\Phi^T_\theta\left(x^{(i)}\right) = y^{(i)},\hspace{1cm}\text{for all }i\in[n].
\end{equation*}
\end{proposition}

The proof may be found in {\bf \Cref{sec: exact.matching}} and relies on several modifications of that presented in \cite{ruiz2023neural}. One may then ask at what level \eqref{eq: neural.ode.drb} differs from \eqref{eq: node}. As it turns out, one important deviation is in the stability of $\theta$ with respect to perturbations in the data (estimate \eqref{eq: stability.estimate}). This can be derived easily in the context of \eqref{eq: neural.ode.drb} as done in \cite{barcena2024optimal}. In the context of \eqref{eq: node} however, the most we are able to say is the following.

\begin{proposition} \label{lemLTA}
Suppose $T>0$, $\upsigma\in \mathscr{C}^0(\mathbb{R})$ and $d\geq n$. Let $y^{(1)},\ldots, y^{(n)}\in\mathbb{R}^d$ be such that
\begin{equation} \label{eq:condrange}
\operatorname{rank}\bigl[\upsigma(y^{(1)}),\dots,\upsigma(y^{(n)})\bigr]=n.
\end{equation}
Then there exist $\varepsilon>0$ and $C>0$ such that for all  $x^{(1)},\dots,x^{(n)}\in\mathbb{R}^d$ with 
\begin{equation*}
\max\limits_{i\in[n]}\|x^{(i)}-y^{(i)}\|_1\leq \varepsilon,   
\end{equation*}
there exist $w\in \mathscr{C}^0([0,T];\mathscr{M}_{d\times d}(\mathbb{R}))$ 
and $(x_i(\cdot))_{i\in[n]}\in\mathscr{C}^1([0, T];\mathbb{R}^{nd})$ satisfying 
\begin{equation*}
\begin{cases}
\dot{x}_i(t) = w(t)\upsigma(x_i(t)) &\text{ for } t\in[0, T]\\
x_i(0) = x^{(i)}\\
x_i(T) = y^{(i)}
\end{cases}
\end{equation*}
for all $i\in[n]$. Moreover,
\begin{equation} \label{eq: stability.estimate}
\|w\|_{\mathscr{C}^0([0,T];\mathscr{M}_{d\times d}(\mathbb{R}))}\leq \frac{C}{T}\max\limits_{i\in[n]}\|x^{(i)}-y^{(i)}\|_1.
\end{equation}
\end{proposition}


\begin{remark}[Genericity] \label{rem: genericity}
Suppose $\upsigma:\mathbb{R}\to\mathbb{R}$ is Lipschitz and strictly increasing, and $d\geq n$. Then \eqref{eq:condrange} holds with probability $1$ for $n$ random points $y^{(i)}$ sampled iid from any probability density $\rho=\frac{\, \mathrm{d}\mu}{\, \mathrm{d} x}$ on $\mathbb{R}^d$. (For instance, this holds if $\upsigma(\cdot)\equiv(\cdot)_+$ and $\rho$ is supported on $(\mathbb{R}_{>0})^d$.)

This is seen by induction on the number of points $n$. The base case $n=1$ trivially holds since for the element-wise extension $\upsigma:\mathbb{R}^d\to\mathbb{R}^d$ of $\upsigma$, $\upsigma_\#\mu$ preserves the zero-measure sets of $\mu$. Suppose that $\upsigma(y^{(1)}),\dots,\upsigma(y^{(n)})$ are almost surely linearly independent for some $1\leq n<d$. 
By absolute continuity of $\upsigma_\#\mu$, the subspace spanned by these points has measure zero, so the next vector $\upsigma(y^{(n+1)})$ lies in the complement with probability $1$, as desired.
\end{remark}

\begin{remark}
\begin{itemize}
    \item Extending \Cref{lemLTA} to account for a bias $b(t)$ as in \eqref{eq: node} is straightforward: we simply write $\dot{z}(t) = \overline{w}(t)\upsigma(z(t))$ where
\begin{equation*}z\coloneqq\begin{bmatrix}
    x\\1
\end{bmatrix}\in\mathbb{R}^{d+1}\quad\text{and}\quad
\overline{w}\coloneqq\begin{bmatrix}
    w & b\\
    0 & 0
\end{bmatrix}\in\mathscr{M}_{(d+1)\times (d+1)}(\mathbb{R}).
\end{equation*}
We now rather need $d\geq n-1$.
    
    \item With regard to \Cref{rem: genericity}, when $\upsigma(\cdot)\equiv(\cdot)_+$ and $\rho$ is supported on $\mathbb{R}^d$, one could simply translate the inputs to $(\mathbb{R}_{>0})^d$ using $b(t)$, but then a bound on $b(t)$ as in \eqref{eq: stability.estimate} cannot hold.
\end{itemize}
\end{remark}

\subsection{Notation}

We use $[n]\coloneqq\{1,\dots,n\}$, $x\wedge y = \min\{x,y\}$, $x \vee y = \max \{x,y\}$. We use $f(x)\lesssim g(x)$ if there exists a finite constant $C>0$ such that $f(x)\leq C g(x)$, and $f(x)\propto g(x)$ if $f(x)=C g(x)$. 

We denote by $\mathscr{S}_{d}^+(\mathbb{R})$ the space of positive definite matrices. On $\mathscr{M}_{d \times d}(\mathbb{R})$, we consider the partial order $P \prec Q$ iff $Q-P$ is positive definite. We also alternate between $a^\top b$ and $\langle a, b\rangle$ depending on presentational convenience.

Given two measures $\mu,\nu$ on $\mathbb{R}^d$, we write $\mu\ll\nu$ to say that $\mu$ is absolutely continuous with respect to $\nu$. If $\mathsf{T}:\mathbb{R}^d\to\mathbb{R}^d$ is measurable, then $\mathsf{T}_{\#}\mu$ denotes the pushforward of $\mu$, defined by $\mathsf{T}_{\#}\mu(A)=\mu(\mathsf{T}^{-1}(A))$ for any measurable $A$. The Gaussian measure with mean $m\in\mathbb{R}^d$ and covariance matrix $\Sigma \in \mathscr{S}_{d}^+(\mathbb{R})$ is denoted by $\mathcal{N}(m, \Sigma)$. 
Finally we write $\mathscr{Z}_\mu=\mu(\mathbb{R}^d)$.

\section{Preliminary lemmas}

We begin with a brief discussion on when 
$\mathsf{L}(\rho_*,\rho(T))=\lim_{\|x\|\to\infty}\rho_*(x)/\rho(T,x)$ can be made finite when $d=1$, purely for illustrative purposes. Recall

\begin{lemma} \label{lem: transport.1d.sol}
    Take $T>0$, a probability density $\rho_{\mathrm{B}}\in\mathscr{C}^0(\mathbb{R})$, and 
    \begin{equation*}
        (w,a,b)(t) = \sum_{k=1}^K (w_k,a_k,b_k)1_{[t_{k-1}, t_{k})}(t) \hspace{1cm} \text{ for } t\in[0, T],
    \end{equation*}
    where $w_k,a_k,b_k\in\mathbb{R}$ with $a_k\neq 0$ and $0=t_0<t_1<\ldots<t_K=T$. Then the unique solution to
    \begin{equation*} 
    \begin{cases}
         \partial_t \rho(t,x) + \partial_x\left( w(t)(a(t)x+b(t))_+ \rho(t,x)\right) = 0 &\text{ on } [0,T]\times\mathbb{R},\\
         \rho(0, x) = \rho_{\mathrm{B}}(x) &\text{ on } \mathbb{R}
    \end{cases}
    \end{equation*}
    is given by
    \begin{align*}
        \rho(t,x) &= \rho_{\mathrm{B}}(x)1_{\left\{\cdot\leq-\frac{b_1}{a_1}\right\}}(x)\\
        &\hspace{1cm}+e^{-w_1a_1 t}\rho_{\mathrm{B}}\left(\left(x+\frac{b_1}{a_1}\right)e^{-w_1a_1t}-\frac{b_1}{a_1}\right)1_{\left\{\cdot>-\frac{b_1}{a_1}\right\}}(x)
    \end{align*}
    for $t\in[0, t_1]$ and
    \begin{align*}
        \rho(t,x) &=  \rho(t_{k-1},x)1_{\left\{\cdot\leq-\frac{b_k}{a_k}\right\}}(x)\\
        &\hspace{1cm}+e^{-w_ka_k t}\rho\left(t_{k-1},\left(x+\frac{b_k}{a_k}\right)e^{-w_ka_kt}-\frac{b_k}{a_k}\right)1_{\left\{\cdot>-\frac{b_k}{a_k}\right\}}(x)
    \end{align*}
    for $t\in[t_{k-1}, t_{k}]$, for all $x\in\mathbb{R}$.
\end{lemma}

\begin{proof}
    The derivation is straightforward: for $t\in[0, t_1)$, $\rho(t,\cdot)$ has a single discontinuity at $x=-b_1/a_1$, and clearly $\rho(t,x) = \rho_{\mathrm{B}}(x)$ for $x\leq -b_1/a_1$. In the case $x>-b_1/a_1$, note that $t\mapsto\rho(t,x(t))$ is constant along the characteristics which solve $\dot{x}(t) = w_1(a_1x(t)+b_1)$, whereupon we deduce the first formula. The  computation is easily repeated for $k\geq1$. 
\end{proof}

\begin{remark}
    If $a_k=0$, then $\rho(t,x) = \rho(t_{k-1}, x-w_k(b_k)_+ t)$ for all $t\in[t_{k-1}, t_k]$ and $x\in\mathbb{R}$, again by the method of characteristics.
\end{remark}

From \Cref{lem: transport.1d.sol} we readily gather that 
\begin{equation*} 
    \rho(T,x)=
    \begin{cases} 
    \alpha^+\rho_{\mathrm{B}}(\alpha^+ x+\beta^+), \hspace{1cm}&\text{if }x\geq M,\\
    \alpha^-\rho_{\mathrm{B}}(\alpha^- x+\beta^-), \hspace{1cm}&\text{if }x\leq -M
    \end{cases}
\end{equation*}
for some $\alpha^\pm>0$, $\beta^\pm\in\mathbb{R}$ and $M>0$, all explicitly depending on $(w, a, b)$---in other words, one can choose $\alpha^\pm, \beta^\pm$ by choosing $(w, a, b)$.
Now consider $\rho_{\mathrm{B}}$ and $\rho_*$ in Gibbs form:
    \begin{equation*}
    \rho_{\mathrm{B}}(x)\propto e^{-U_{\mathrm{B}}(x)}, \hspace{1cm}\rho_*(x)\propto e^{-U_*(x)},
    \end{equation*}
with $U_{\mathrm{B}},U_*\in \mathscr{C}^0(\mathbb{R};\mathbb{R}_{\geq0})$. We see that 
    \begin{equation*}
        \frac{\rho_*(x)}{\rho(T,x)} \propto \exp\left(U_{\mathrm{B}}(\alpha^\pm x+\beta^\pm)-U_*(x)\right)
    \end{equation*}
for all $\|x\|\geq M$. So $\mathsf{L}(\rho_*,\rho(T))$ defined in \eqref{eq:condtails} is finite if and only if 
\begin{equation} \label{eq: gibbs.potentials}
\lim_{\|x\|\to\infty}U_{\mathrm{B}}(\alpha^\pm x+\beta^\pm)-U_*(x)\in[-\infty,+\infty).
\end{equation}
One sees the constraints that \eqref{eq: gibbs.potentials} entails. For instance, one can consider as input the Laplace distribution---corresponding to $U_{\mathrm{B}}(x) = |x-m|/\sigma$ with parameters $m_{\mathrm{B}}\in\mathbb{R}$ and $\sigma_{\mathrm{B}}>0$---, and as target the Gaussian---corresponding to $U_*(x) = (x-m_*)^2/2\sigma_*^2$ with $m_*\in\mathbb{R}$ and $\sigma_*>0$---, but not the converse.

\Cref{lem: transport.1d.sol} can be generalized to the case $d\geq 1$ without difficulty.

\begin{lemma} \label{lem: transport.multid.sol}
    Take $T>0$, a probability density $\rho_{\mathrm{B}}\in\mathscr{C}^0(\mathbb{R}^d)$, and
    \begin{equation*}
        (w,a,b)(t) = \sum_{k=1}^K (w_k, a_k, b_k)1_{[t_{k-1}, t_k)}(t) \hspace{1cm} \text{ for } t\in[0, T],
    \end{equation*}
    where $w_k, a_k\in\mathbb{R}^d$, $b_k\in\mathbb{R}$, and $0=t_0<t_1<\ldots<t_K=T$. Then the unique solution to 
    \begin{equation} \label{eq: cont.eq.clear}
        \begin{cases}
            \partial_t \rho(t,x) + \mathrm{div}\left(w(t)\left(a(t)^\top x + b(t)\right)_+\rho(t,x)\right) = 0 &\text{ on } [0, T]\times\mathbb{R}^d,\\
            \rho(0, x) = \rho_{\mathrm{B}}(x) &\text{ on } \mathbb{R}^d
        \end{cases}
    \end{equation}
    is given by
    \begin{align*}
        \rho(t,x) &= \rho(t_{k-1}, x)1_{\big\{\langle a_k, \cdot\rangle+b_k\leq 0\big\}}(x) \\
        &\hspace{1cm}+ e^{-w_k^\top a_k(t-t_{k-1})}\rho\left(t_{k-1}, \mathsf{A}_{k}(x)\right) 1_{\big\{\langle a_k, \cdot\rangle+b_k>0\big\}}(x)
    \end{align*}
    if $w_k^\top a_k\neq0$, where 
    \begin{equation*}
    \mathsf{A}_{k}(x) = e^{-(t-t_{k-1})w_ka_k^\top}x-\frac{b_k}{w_k^\top a_k} \left(1-e^{-(t-t_{k-1}) w_k^\top a_k}\right)w_k,
    \end{equation*}
    and
    \begin{align*}
        \rho(t,x) = \rho(t_{k-1}, x)1_{\big\{\langle a_k, \cdot\rangle+b_k\leq 0\big\}}(x)+\rho(t_{k-1}, \mathsf{B}_k(x))1_{\big\{\langle a_k, \cdot\rangle+b_k>0\big\}}(x)
    \end{align*}
    if $\langle w_k, a_k\rangle=0$, where
     \begin{equation*}
   \mathsf{B}_k(x) = e^{-(t-t_{k-1})w_ka_k^\top}x-e^{-(t-t_{k-1})w_ka_k^\top}(t-t_{k-1})b_kw_k,   
    \end{equation*}
    for all $t\in[t_{k-1}, t_k]$ and for all $x\in\mathbb{R}^d$.
\end{lemma}

\begin{proof}
We again start by $t\in[0, t_1)$. For $x\in\{\langle a_1,\cdot\rangle+b_1\leq 0\}$, we clearly have $\rho(t,x) = \rho_{\mathrm{B}}(x)$. On $\{\langle a_1,\cdot\rangle+b_1>0\}$, $t\mapsto \rho(t,x(t))\exp(w_1^\top a_1t)$ is constant along the characteristic surfaces parametrized by solutions to
\begin{equation*}
    \dot{x}(t) = w_1a_1^\top x(t) + w_1b_1.
\end{equation*}
Suppose $w_1^\top a_1\neq 0$. Since 
\begin{equation*}
 e^{t w_1a_1^\top} = I_d + \frac{(e^{tw_1^\top a_1}-1)}{w_1^\top a_1}w_1a_1^\top,   
\end{equation*}
we have 
\begin{align*}
    e^{tw_1a_1^\top}\left(\int_0^t e^{-sw_1a_1^\top} \, \mathrm{d} s\right) w_1b_1= \frac{b_1}{ w_1^\top a_1}\left(1-e^{-tw_1^\top a_1}\right) e^{tw_1a_1^\top} w_1.
\end{align*}
Thus 
\begin{equation*}
 x(t) = e^{tw_1a_1^\top}x(0) +  \frac{b_1}{ w_1^\top a_1}\left(1-e^{-tw_1^\top a_1}\right) e^{tw_1a_1^\top} w_1,
\end{equation*}
and so
\begin{equation*}
   \rho(t,x) = e^{-tw_1^\top a_1}\rho_{\mathrm{B}}\left(e^{-tw_1a_1^\top}x-\frac{b_1}{w_1^\top a_1}\left(1-e^{-tw_1^\top a_1}\right)w_1\right) 
\end{equation*}
for $x\in\{\langle a_1,\cdot\rangle+b_1>0\}$ when $w_1^\top a_1\neq0$. 

Now suppose $w_1^\top a_1=0$. Then $w_1a_1^\top$ is nilpotent, so $e^{tw_1a_1^\top} = I_d + tw_1a_1^\top$. 
Thus 
\begin{equation*}
  x(t) = e^{tw_1a_1^\top}x(0) + t b_1w_1,  
\end{equation*}
and so
\begin{equation*}
\rho(t,x) = \rho_{\mathrm{B}}\left(e^{-tw_1a_1^\top}x-e^{-tw_1a_1^\top}tb_1w_1\right) 
\end{equation*}
for $x\in\{\langle a_1,\cdot\rangle+b_1>0\}$.

The computations are easily repeated for $k\geq 1$.
\end{proof}

The following comparison principle will be useful.

\begin{lemma} \label{lem:comparison}
Let $\rho_1$ and $\rho_2$ be two solutions to \eqref{eq: cont.eq.clear} corresponding to the same $\theta$, and to data $\rho_1(0, \cdot)$ and $\rho_2(0, \cdot)$ respectively. 
Suppose $\rho_1(0, x)<\rho_2(0, x)$ for all $x\in\mathbb{R}^d$. Then $\rho_1(t, x)<\rho_2(t, x)$ for all $t\in[0, T]$ and $x\in\mathbb{R}^d$.  
\end{lemma}

\begin{proof}
Set $f(t,x)\coloneqq\rho_2(t,x)-\rho_1(t,x)$ and $v(t,x)\coloneqq w(t)(a(t)^\top x+b(t))_+$. Along the characteristics $\dot{x}(t) = v(t,x(t))$,
\begin{align*}
    \frac{\, \mathrm{d}}{\, \mathrm{d} t} f(t,x(t)) &= \partial_t f(t,x(t)) + \nabla_x f(t,x(t))^\top v(t,x(t))\\
    &=-f(t,x(t))\,\mathrm{div}_x v(t,x(t)).
\end{align*}
Thus
\begin{equation*}
    f(t,x(t))= f(0,x(0))\exp\left(-\int_0^t\mathrm{div}_x v(s,x(s))\, \mathrm{d} s\right)
\end{equation*}
for all $t\in[0, T]$ and $x(0)\in\mathbb{R}$. We may conclude.
\end{proof}

We also use the following elementary lemma.

\begin{lemma} \label{lem:variancecomparison}

Let $\rho_1$ and $\rho_2$ be two nonnegative and integrable functions that are proportional to Gaussian densities with covariance matrices $\Sigma_1$ and $\Sigma_2$ respectively. Suppose
\begin{equation*}
\left\langle (\Sigma_2^{-1}-\Sigma_1^{-1})e_i,e_i\right\rangle<0,\hspace{1cm}\text{ for some }i\in[d].
\end{equation*}
Then, for any $M>0$, there exists $M(i)>0$ such that 
\begin{equation*}
\rho_2(x)>\rho_1(x) \hspace{1cm}\text{ in } \left\{x\in\mathbb{R}^d\colon |x_i|>M(i) \text{ and } |x_j|<M \text{ for } j\neq i\right\}.
\end{equation*}
If $\Sigma_2^{-1}\prec \Sigma_1^{-1}$, then there exists some $\overline M>0$ such that 
\begin{equation*}
\rho_2(x)>\rho_1(x)\hspace{1cm}\text{ for all } \|x\|>\overline{M}.
\end{equation*}
\end{lemma}
\begin{proof}
Write
    \begin{equation*}
    \frac{\rho_2(x)}{\rho_1(x)}\propto e^{-\frac{1}{2}\left\langle\left(\Sigma_2^{-1}-\Sigma_1^{-1}\right)x,x\right\rangle+O(\|x\|)}.
    \end{equation*}
 The proof follows.
\end{proof}

\section{Proofs}

\subsection{Proof of \Cref{thm:KLcontrolmultid}} \label{sec: proof.thm1}

\begin{proof}[Proof of \Cref{thm:KLcontrolmultid}]
    Fix $\varepsilon_0>0$; $\theta=(w,a,b)$ takes the form
\begin{equation*}
    \theta(t)=\bar\theta(t)1_{\left[0,\frac{T}{2}\right]}(t)+\sum_{k=1}^{2d}\theta_k 1_{[T_{k-1},T_k]}(t),
\end{equation*}
where $T_k\coloneqq\frac{T}{2}+\frac{kT}{4d}$, and $\bar\theta$ is piecewise constant with at most $\left\lceil \frac{2R(\varepsilon)}{h(\varepsilon)}\right\rceil^d(d+10)$ switches and such that
\begin{equation} \label{eq:TVproof}
 \left\| \rho\left(\frac{T}{2}\right)-\rho_*\right\|_{L^1(\mathbb{R}^d)}\leq\varepsilon_0.
 \end{equation}
Such a $\bar\theta$ exists by virtue of \cite[Theorem 1]{ruiz2024control}. We seek to construct $\theta_1,\ldots,\theta_{2d}$ so that 
\begin{equation*}
    \mathsf{L}(\rho_*,\rho(T)) = \lim_{\|x\|\to+\infty} \frac{\rho_*(x)}{\rho(T,x)} = 0.
\end{equation*}

\subsubsection*{Step 1.}
Thanks to \Cref{lem: transport.multid.sol}, there exists a finite $n\geq 1$, as well as polyhedra $P_i\subset\mathbb{R}^d$ and coefficients $(\alpha_i,A_i,\beta_i)\in\mathbb{R}_{>0}\times\mathscr{M}_{d\times d}(\mathbb{R})\times \mathbb{R}$ with $\mathrm{spec}(A_i)\subset(0,+\infty)$ for $i\in[n]$, such that
\begin{equation*}
\rho\left(\frac{T}{2},x\right)=\sum_{i=1}^n\alpha_i\rho_{\mathrm{B}}(A_i x+\beta_i)1_{P_i}(x) \hspace{1cm} \text{ for all } x\in\mathbb{R}^d.
\end{equation*}
Moreover, the polyhedra $P_i$ form a partition of $\mathbb{R}^d$, which implies $\rho\left(T/2\right)>0$. Each $\rho_i\coloneqq\alpha_i\rho_{\mathrm{B}}(A_i \cdot+\beta_i)$ is a Gaussian function on $\mathbb{R}^d$, and the Gaussian probability density $\mathscr{Z}_i^{-1}\rho_i$ has covariance matrix
\begin{equation*}
\Sigma_i\coloneqq
A_i^{-1}\Sigma_{\mathrm{B}}\left(A_i^{-1}\right)^\top.
\end{equation*} 
Take $0<\underline{\sigma}^2<\min_{i\in[n]}\mathrm{spec}(\Sigma_i)$ and define
\begin{equation} \label{eq:defrholb}
\underline{\rho}(x)\coloneqq \alpha \exp\left(-\frac{\|x\|^2}{2\underline{\sigma}^2}\right)
\end{equation}
for $x\in\mathbb{R}^d$, where $\alpha>0$ is a constant to be specified. 
Then $\Sigma_i^{-1}\prec I_{d}/\underline{\sigma}^2$ for all $i\in[n]$, so by virtue of \Cref{lem:variancecomparison}, there exists $M(i)>0$ such that
\begin{equation*}
    \underline{\rho}(x)<\rho_i(x)\hspace{2.2cm}\text{ for all } \quad \|x\|> M({i}).
\end{equation*}
Taking $\overline M\geq M \vee \max_{i\in[n]}M(i)$ it follows that \eqref{eq:condlowgaussmultid} and $\underline{\rho}(x)<\rho\left(T/2, x\right)$ hold for all $\|x\|>\overline{M}$.
Choosing $\alpha$ small enough, we can deduce 
\begin{equation} \label{eq: lnTm2d}
    \underline{\rho}(x)<\rho\left(\frac{T}{2}, x\right) \hspace{1cm} \text{ for all } x\in\mathbb{R}^d,
\end{equation}
and $\overline M$ large enough to deduce
 \begin{equation}\label{eq:outermassmultid}
\int_{[-\overline M,\overline M]^d}\rho_*(x)\, \mathrm{d} x \wedge \int_{[-\overline M,\overline M]^d}\rho\left(\frac{T}{2}, x\right)\, \mathrm{d} x>1-\varepsilon_0. 
\end{equation}   

\subsubsection*{Step 2.}
We will choose $\theta_{1},\dots,\theta_{2d}$ so that
\begin{equation}\label{eq: bound rholb}
 \rho_\bullet(x)<\underline{\rho}(T, x),\hspace{1cm}\text{in  }\mathbb{R}^d\setminus[-\overline{\overline{M}},\overline{\overline{M}}]^d, 
 \end{equation} 
for some larger $\overline{\overline{M}}\geq\overline{M}$, where $\underline{\rho}(t, \cdot)$ is the unique solution to \eqref{eq: cont.eq} in $\left[T/2,T\right]$ with data $\underline{\rho}\left(T/2, \cdot\right)=\underline{\rho}(\cdot)$. To this end, take 
\begin{align*}
    \theta_{2k-1} &\coloneqq(\omega e_k,e_k,-\overline{M}),\\
    \theta_{2k} &\coloneqq(-\omega e_k,-e_k,-\overline{M}),
\end{align*}
where $\omega>0$ is to be chosen, for $k\in[d]$. 
By virtue of \Cref{lem: transport.multid.sol}, the solution reads
\begin{align*}
        \underline{\rho}(t,x) =& \underline{\rho}(T_{2k-2}, x)1_{\left\{x\colon x_k\leq \overline{M}\right\}}(x) \\
        &+ e^{-\omega(t-T_{2k-2})}\underline{\rho}\left(T_{2k-2}, \mathsf{A}_{2k-2}(t,x)\right) 1_{\left\{x\colon x_k> \overline{M}\right\}}(x)
    \end{align*}
if $t\in[T_{2k-2},T_{2k-1}]$, and 
\begin{align*}
        \underline{\rho}(t,x) =& \underline{\rho}(T_{2k-1}, x)1_{\left\{x\colon x_k\geq- \overline{M}\right\}}(x) \\
        &+ e^{-\omega(t-T_{2k-1})}\underline{\rho}\left(T_{2k-1}, \mathsf{A}_{2k-1}(t,x)\right) 1_{\left\{x\colon x_k< -\overline{M}\right\}}(x)
    \end{align*}
if $t\in[T_{2k-1},T_{2k}]$, where
    \begin{equation*}
    \mathsf{A}_{j}(t, x) = e^{-\omega(t-T_{j})e_ke_k^\top}x+(-1)^{j}\overline{M}\left(1-e^{-\omega(t-T_{j})}\right)e_k
    \end{equation*}
    for $j\in\{2k-2,2k-1\}$, for all $k\in[d]$ and $x\in\mathbb{R}^d$. 
    It ensues that 
    \begin{align} \label{eq: rholb}
        \underline{\rho}(T, x) =& \underline{\rho}\left(\frac{T}{2}, x\right)1_{[-\overline{M},\overline{M}]^d}(x) \nonumber \\
        &+\sum_{k=1}^d \sum_{\ell=0}^1\underline{\rho}_{1}^{k\ell}(x) 1_{\mathscr{S}_{k\ell}}(x)+\sum_{q\in\{0,1\}^{d}}\underline{\rho}_{2}^q(x) 1_{\mathscr{Q}_q}(x),
    \end{align} 
    where $\underline{\rho}_{1}^{k\ell}$, $\underline{\rho}_{2}^k$ are Gaussian functions on $\mathbb{R}^d$, and $\mathscr{Q}_q$, $\mathscr{S}_{k\ell}$ are defined by
    \begin{equation*}
       \mathscr{Q}_q\coloneqq\left\{x \in \mathbb{R}^d : (-1)^{q_j}x_j>\overline{M}\quad\text{for all }j\in[d]\right\}  
     \end{equation*}
    and 
    \begin{equation*}
    \mathscr{S}_{k\ell}\coloneqq\left\{(-1)^{\ell}x_k>\overline{ M}\right\}\setminus\bigcup_{q\in\{0,1\}^d}\mathscr{Q}_q.
    \end{equation*}   
    The Gaussian density $\mathscr{Z}^{-1}\underline{\rho}_{1}^{k\ell}$ has covariance
    \begin{align*}
    \Sigma_{1,k\ell}&=\left(e^{-\frac{\omega T}{4d}e_ke_k^\top}\right)^{-1}\underline{\sigma}^2I_d\left(\left(e^{-\frac{\omega T}{4d}e_ke_k^\top}\right)^{-1}\right)^\top\notag\\
    &=\underline{\sigma}^2\left( I_d+(e^{-\frac{\omega T}{4d}}-1)e_k e_k^\top\right)^{-1}\left( I_d+(e^{-\frac{\omega T}{4d}}-1)e_ke_k^\top\right)^{-1}\notag\\
    &=\underline{\sigma}^2\left( I_d+(e^{\frac{\omega T}{4d}}-1)e_k e_k^\top\right)\left( I_d+(e^{\frac{\omega T}{4d}}-1)e_ke_k^\top\right)\\
    &=\underline{\sigma}^2\left(I_d+(e^{\frac{\omega T}{2d}}-1)e_k e_k^\top\right),
\end{align*}
    Taking 
    \begin{equation}\label{eq: omega}
    \omega >\frac{2d}{T}\log\left(\frac{\sigma_\bullet^2}{\underline{\sigma}^2}\right),
    \end{equation}
    we find for each $k\in[d]$, $\ell\in\{0,1\}$, that
    \begin{equation*}
    \left\langle\left(\Sigma_{1,k\ell}^{-1}-\frac{1}{\sigma_\bullet^2}I_d\right)e_k,e_k\right\rangle=\frac{e^{-\frac{\omega T}{2d} }}{\underline{\sigma}^{2}}-\frac{1}{\sigma_\bullet^{2}}<0.
\end{equation*}
A similar computation for the covariance of $\mathscr{Z}^{-1}\underline{\rho}_{2}^q$ gives  $\Sigma_{2,q}=\underline{\sigma}^2e^{\frac{\omega T}{2d}}I_d.$
We see that $\Sigma_{2,q}^{-1}\prec I_d/\sigma_\bullet^2$ for all $q\in\{0,1\}^d$ when $\omega$ satisfies \eqref{eq: omega}. 
By \Cref{lem:variancecomparison} there exists some $M(0)\geq \overline{M}$ such that 
\begin{equation*}
\rho_\bullet(x)<\underline{\rho}_2^q\left(x\right)\hspace{1cm}\text{ for all } \|x\|>M(0),
\end{equation*} 
for $q\in\{0,1\}^d$. Also by \Cref{lem:variancecomparison}, for all $k\in[d]$ and $\ell\in\{0, 1\}$ there exists $M(k)>0$ such that \begin{equation*}
    \rho_\bullet(x)<\underline{\rho}_{1}^{k\ell}\left(x\right)\hspace{1cm}\text{ in } \left\{x\in\mathbb{R}^d\colon |x_k|>M(k) \text{ and } |x_j|<M(0)\text{ for } j\neq k \right\}.
\end{equation*}
Thanks to \eqref{eq: rholb}, taking $\overline{\overline{M}}\coloneqq\max_{k\in\{0,\ldots,d\}} M(k)$ ensures \eqref{eq: bound rholb}. 

\subsubsection*{Step 3.}

Combining \eqref{eq:condlowgaussmultid}, \eqref{eq: lnTm2d}, \eqref{eq: bound rholb} and \Cref{lem:comparison}, we find
\begin{equation*}
\rho_*(x)<\rho(T, x),\hspace{1cm}\text{ in }\mathbb{R}^d\setminus[-\overline{\overline{M}},\overline{\overline{M}}]^d.
\end{equation*} 
In particular, this implies $\mathsf{L}(\rho_*,\rho(T))=0$, 
where $\mathsf{L}$ is defined in \eqref{eq:condtails}.
We moreover have
\begin{equation*}  
\min_{x\in [-\overline{\overline{M}},\overline{\overline{M}}]^d}\rho\left(T,x\right)>\min_{x\in [-\overline{\overline{M}},\overline{\overline{M}}]^d}\underline{\rho}\left(T,x\right)>0
\end{equation*}
by \eqref{eq:defrholb}, \eqref{eq: lnTm2d} and \eqref{eq: rholb},  so we conclude that \begin{equation*}
\sup_{x\in\mathbb{R}^d}\frac{\rho_*(x)}{\rho(T,x)}<+\infty.
\end{equation*}
On the other hand, $\rho(T, x) = \rho(T/2, x)$ for $x\in[-\overline{M},\overline{M}]^d$, which, combined with \eqref{eq:TVproof} and \eqref{eq:outermassmultid}, yields
\begin{align*}
    \left\|\rho(T)-\rho_*\right\|_{L^1(\mathbb{R}^d)}\leq
    &\int_{[-\overline{M},\overline{M}]^d} \left|\rho\left(\frac{T}{2},x\right)-\rho_*\left(x\right)\right|\, \mathrm{d} x\\
    &+\int_{\mathbb{R}^d\setminus[-\overline{M},\overline{M}]^d}\rho(T,x)\, \mathrm{d} x\\&+\int_{\mathbb{R}^d\setminus[-\overline{M},\overline{M}]^d}\rho_*(x)\, \mathrm{d} x<\,3\varepsilon_0.
\end{align*}
Taking $\varepsilon_0$ small enough and applying \Cref{lem:revpinsker1}, we conclude. 
\end{proof}

\subsection{Proof of \Cref{thm:revKLcontrolmultid}} \label{sec: proof.thm2}

\begin{proof}[Proof of \Cref{thm:revKLcontrolmultid}]
The proof is almost identical to that of \Cref{thm:KLcontrolmultid}--we only discuss the differences. 

In Step 1, we take $\overline{\sigma}^2>\max_{i\in[n]}\mathrm{spec}(\Sigma_i)$ instead of $\underline{\sigma}^2$, defining $\overline{\rho}(x)$ as $\underline{\rho}(x)$ in \eqref{eq:defrholb}, as to have $\overline{\rho}(x)>\rho(T/2,x)$ for all $x\in\mathbb{R}^d$. 

In Step 2, we take 
\begin{align*}
    \theta_{2k-1} &=(-\omega e_k,e_k,-\overline{M}),\\
    \theta_{2k} &=(\omega e_k,-e_k,-\overline{M}),
\end{align*}
for $k\in[d]$, with $\omega >2d\log\left(\sigma_\bullet^2/\overline{\sigma}^2\right)/T.$
Then the same argument as in Step 2 yields $\overline{\rho}(T,x)<\rho_\bullet(x)$ for all $x\in \mathbb{R}^d\setminus[-\overline{\overline{M}},\overline{\overline{M}}]^d$,
for some $\overline{\overline{M}}\geq \overline{M}$.

In Step 3, \Cref{lem:comparison} then yields $\rho(T, x)<\rho_*(x)$ for all $x\in \mathbb{R}^d\setminus[-\overline{\overline{M}},\overline{\overline{M}}]^d$. This implies 
$\sup_{x\in\mathbb{R}^d}\frac{\rho(T,x)}{\rho_*(x)}<+\infty$ because $\rho_*>0$ in $\mathbb{R}^d$. The conclusion follows thereupon.    
\end{proof}

\subsection{Beyond Gaussians} \label{sec: on.remark.gauss}

In this section we comment on how the extension discussed in \Cref{rem:extKL} can be accounted for in the above proofs.
Fix $\varepsilon_0>0$. From \cite{azagra_global_2013} we deduce the existence of a convex $U_\bullet\in\mathscr{C}^2(\mathbb{R}^d)$ with $U_{\bullet}\leq \operatorname{conv}U_*$ such that 
\begin{equation*}
    \|U_{\bullet}-\operatorname{conv}U_*\|_{\mathscr{C}^0(\mathbb{R}^d)}<\varepsilon_0.
\end{equation*}
For any $A\in\mathscr{M}_{d\times d}(\mathbb{R})$ with $\operatorname{spec}(A)\subset(0,+\infty)$, the $\lambda_A>0$ from \eqref{eq:condlowmultidB} gives
\begin{align*}
    &\lim_{\|x\| \to +\infty}\frac{U_\bullet(\lambda_A x)}{U_{\mathrm{B}}(A x)}\\
    &=\lim_{\|x\| \to +\infty}\frac{U_\bullet(\lambda_A x)-\operatorname{conv}U_*(\lambda_A x)}{U_{\mathrm{B}}(A x)}+\frac{\operatorname{conv}U_*(\lambda_A x)}{U_{\mathrm{B}}(A x)}=+\infty.
\end{align*}
We now show that $ e^{-U_{\bullet}}$ satisfies a comparison principle similar to \cref{lem:variancecomparison}. Namely, given any $P,Q$ diagonal with positive entries and $\beta\in\mathbb{R}^d$, we prove that
\begin{equation}\label{eq:compubullet1}
\lim_{|x_i| \to +\infty }U_{\bullet}(Px+\beta)-U_{\bullet}(Qx)=+\infty\hspace{1cm}\text{if }P_{i}>Q_{i},
\end{equation}
and 
\begin{equation} \label{eq:compubullet2}
\lim_{\|x\| \to +\infty}U_{\bullet}(Px+\beta)-U_{\bullet}(Qx)=+\infty\hspace{1cm}\text{if }P\succ Q.
\end{equation}
If $\varepsilon_0>0$ is small enough, thanks to \eqref{eq:condlowmultid*} we have
\begin{equation*}
    U_\bullet(x)\gtrsim\|x\|\hspace{2.8cm} \text{for all }\|x\|>M_*.
\end{equation*}
By virtue of $U_\bullet\in\mathscr{C}^2$ and convexity, there exists $M_\bullet>0$ such that for all $i\in[d]$,
\begin{equation*}
    \partial_{x_i}U_{\bullet}(x)|x_i|\gtrsim x_i\hspace{2cm}\text{ for all } \|x\|>M_\bullet.
\end{equation*}
Take $\|x\|>M_\bullet$ and, without loss of generality, suppose $x\in(\mathbb{R}_{>0})^d$, so that 
\begin{equation*}
    \partial_{x_i}U_{\bullet}(x)\gtrsim 1
\end{equation*}
for all $i\in[d]$.
We now decompose
\begin{align*}
    U_{\bullet}(Px)-U_{\bullet}(Qx)&=U_{\bullet}(Px)-U_{\bullet}\left(P_1x_1,\dots,P_{d-1}x_{d-1},Q_dx_d\right)\\
&\hspace{1cm}+\cdots+U_{\bullet}\left(P_1x_1,Q_2x_2,\dots,Q_dx_d\right)-U_{\bullet}\left(Qx\right)  \\
&=\sum_{i=1}^d\int_{(Qx)_i}^{(Px)_i}\partial_{x_i}U_{\bullet}((Px)_{<i},t,(Qx)_{>i})dt,
\end{align*}
where
\begin{equation*}
    ((Px)_{<i},t,(Qx)_{>i})=((Px)_1,\dots,(Px)_{i-1},t,(Qx)_{i+1},\dots,(Qx)_n).
\end{equation*}
Suppose $P_{i}>Q_{i}$ for all $i$. We deduce:
\begin{equation*}
    U_{\bullet}(Px)-U_{\bullet}(Qx)\gtrsim \sum_{i=1}^d(P_{i}-Q_{i})x_i\gtrsim \|x\|.
\end{equation*}
Now suppose $P_{i}>Q_{i}$ for some $i\in[d]$. If we fix $x_j$ for all $j\neq i$, we deduce
\begin{equation*}
     U_{\bullet}(Px)-U_{\bullet}(Qx)      \gtrsim (P_{i}-Q_{i})x_i-1\gtrsim x_i-1.
\end{equation*}
These inequalities give \eqref{eq:compubullet1} and \eqref{eq:compubullet2}. 

It remains to outline the adjustments in the proof of \cref{thm:KLcontrolmultid} required to generalize the result. The  generalization of \cref{thm:revKLcontrolmultid} is completely analogous.

In Step 1, since $\beta_i\in\mathbb{R}^d$ and $\alpha_i>0$ are fixed for all $i\in[n]$, 
 thanks to  \eqref{eq:condlowmultidB} we can find $\underline{\lambda}>\max_{i\in[n]}\lambda_{A_i}$ and $\overline M>0$ such that \begin{equation*}
    U_\bullet (\underline{\lambda}x)>\max_{i\in[n]}U_{\mathrm{B}}(A_ix+\beta_i)-\log\alpha_i \hspace{1.4cm}\text{for all }\|x\|>\overline M.
\end{equation*}
We define $\underline{\rho}(x)=\alpha e^{-U_\bullet(\underline{\lambda} x)}$ for all $x\in\mathbb{R}^d$, which satisfies
\begin{equation*} 
   \underline{\rho}(x)<\min_{i\in[n]}\alpha_i\rho_{\mathrm{B}}(A_i x+\beta_i)\leq\rho\left(\frac{T}{2}, x\right) \hspace{1cm} \text{ for all } \quad \|x\|> \overline M.
\end{equation*}
Choosing $\alpha$ small enough, we can deduce $\underline{\rho}(x)<\rho\left(T/2, x\right)$ for all $x\in\mathbb{R}^d$.

In Step 2, we take the same controls $\theta_1,\dots,\theta_{2d}$ that give \eqref{eq: rholb}, now with \begin{align*}
        \underline{\rho}_{1}^{k\ell}(x)&=e^{-\frac{\omega T}{4d}}\underline{\rho}\left(\left(I_d+(e^{-\frac{\omega T}{4d}}-1)e_ke_k^\top\right)x+
        \beta_{1,k\ell}\right), \\
        \underline{\rho}_{2}^q(x)&=e^{-\frac{\omega T}{4}}\underline{\rho}\left(e^{-\frac{\omega T}{4d}}x+\beta_{2,q}\right)
    \end{align*}
     for some 
     $ \beta_{1,k\ell},\beta_{2,q}\in\mathbb{R}^d$. 
Take $\omega>4d\log\underline{\lambda}/T$, which gives $\underline{\rho}(T,x)>e^{-U_\bullet(x)}$ for $\|x\|$ large enough, thanks to \eqref{eq:compubullet1} and \eqref{eq:compubullet2}. Since $U_\bullet\leq U_*$ by construction,   \cref{lem:comparison} ensures the existence of $\overline{\overline{M}}\geq\overline{M}\vee M_\bullet$ such that \begin{equation*}
\rho_*(x)<\underline{\rho}(T, x)<\rho(T,x),\hspace{1cm}\text{in  }\mathbb{R}^d\setminus[-\overline{\overline{M}},\overline{\overline{M}}]^d.
\end{equation*} 

\subsection{Proof of \Cref{prop: xlogx}} \label{sec: proof.propxlogx}

\begin{proof}[Proof of \Cref{prop: xlogx}]

For $x\in\mathbb{R}$ and $t\in\mathbb{R}$, consider
\begin{equation*} 
x(t)=
\begin{cases}
x^{e^t} &\text{ if } x> 1\\
x &\text{ if } x\leq 1.
\end{cases}
\end{equation*}
This function solves
\begin{equation*} 
\begin{cases}
    \dot{x}(t)=(|x(t)|\log x(t)_+)_+& t\in\mathbb{R}\\
    x(0)=x.
    \end{cases}
\end{equation*}
It is in fact the unique solution. Indeed, the local Lipschitz condition guarantees uniqueness on the maximal interval of existence, and $x(t)$ is defined for all $t\in\mathbb{R}$. 

One can thus define the solution $\rho(t,x)$ to \eqref{eq: log.transport} uniquely in the weak sense.  
We focus on the behavior of the solution $\rho(t,\cdot)$ on $\{x\colon x>1\}$. By the method of characteristics, we can write
    \begin{equation*}
        \frac{\, \mathrm{d}}{\, \mathrm{d} t}\rho(t, x(t))=-(\log x(t)+1)\rho(t, x(t)),
    \end{equation*}
    which translates into
    \begin{equation*}
        \int_0^t\frac{\, \mathrm{d}}{\, \mathrm{d} s}\rho(s, x(s))\rho(s, x(s))^{-1}\, \mathrm{d} s=\int_0^t-(\log(x)e^{s}+1)\, \mathrm{d} s.
    \end{equation*}
    Whereupon,
    \begin{equation*}
        \log \rho(t, x(t))-\log \rho_{\mathrm{B}}(x)=-\left(\log(x)(e^{t}-1)+t\right),
    \end{equation*}
    whence
    \begin{equation*}
        \rho(t, x(t))=\rho_{\mathrm{B}}(x)\exp\left(-\log x^{(e^t-1)}-t\right)
    \end{equation*}
    and
    \begin{equation*}
        \rho(t, x^{e^t})=e^{-t}\rho_{\mathrm{B}}(x)x^{1-e^t}.
    \end{equation*}
    Taking $y=x^{e^t}$, we obtain $y^{e^{-t}}=x$ and 
    \begin{equation*}
        \rho(t, y)=e^{-t}\rho_{\mathrm{B}}\left(y^{e^{-t}}\right)y^{e^{-t}-1}.
    \end{equation*}
    Upon considering the case $x\leq 1$, we conclude
    \begin{equation*}
    \rho(t, x)=\rho_{\mathrm{B}}(x)1_{\{x\colon x\leq 1\}}(x)+e^{-t}\rho_{\mathrm{B}}\left(x^{e^{-t}}\right)x^{e^{-t}-1}1_{\{x\colon x> 1\}}(x).
    \end{equation*}
    We thus see that
    \begin{align*}
        \lim_{x\to +\infty}\frac{\rho(t,x)}{\rho_{\mathrm{B},q}(x)}\propto e^{-t}\lim_{x\to +\infty}\frac{\exp\left(-x^{pe^{-t}}\right)x^{e^{-t}-1}}{\exp(-x^q)}.
    \end{align*}
    It therefore suffices to take
    \begin{equation*}
        T_q>\log\left(\frac{p}{q}\right),
    \end{equation*}
    which is positive if $p>q$, to ensure 
    \begin{equation*}
        \lim_{x\to +\infty}\frac{\rho(T_q,x)}{\rho_{\mathrm{B},q}(x)}<+\infty.\qedhere
    \end{equation*}    
\end{proof}

\begin{remark}
Observe that by simply changing the sign of the vector field (or reversing time), we obtain the same result for $q\geq p$.
\end{remark}

\subsection{Proof of \texorpdfstring{\Cref{thm:exactcontrol}}{Proof of Theorem \ref{thm:exactcontrol}}}\label{sec: exact.matching}

Ultimately, $\theta=(w,b)$ takes the form
\begin{equation*}
    \theta(t)=\sum_{k=1}^{4n+3}\theta_k 1_{[T_{k-1},T_k]}(t),
\end{equation*}
where $\theta_k=(w_k, b_k)\in\mathscr{M}_{d\times d}(\mathbb{R})\times\mathbb{R}^d$ are constant and determined throughout the proof, whereas $T_k\coloneqq\frac{kT}{4n+3}$. 

\subsubsection*{Step 1.}
We begin by taking $\theta_1=(w_1,b_1)$ with
\begin{equation*}
w_1=0, \hspace{1cm} b_1=\beta_1 {\bf1}
\end{equation*}
for $\beta_1>0$. Then for all $i\in[n]$, 
\begin{equation*}
 x^{(i)}(t)=x^{(i)}(0)+t\beta_1\, {\bf1}  \hspace{1cm} \text{ for } t\in[0, T_1].
\end{equation*}
Taking $\beta_1$ large enough, we find
\begin{equation*}
    x_k^{(i)}(T_1)>0 \hspace{1cm}\text{ for all }k\in[d].
\end{equation*}
We now take $\theta_2=(w_2,b_2)$ with
\begin{equation*} 
w_2=\alpha_1\begin{bmatrix} 0 &\ldots & 0 & \ldots & 0\\ v_1 &\ldots & v_k & \ldots & v_d\\ 0 &\ldots & 0 & \ldots & 0\\ \vdots & \ddots & \vdots & \ddots & \vdots \\ 0 &\ldots & 0 & \ldots & 0
\end{bmatrix}, \hspace{1cm} b_2=0,
\end{equation*}
where $\alpha_1>0$ and $v=(v_1,\ldots, v_d)\in(\mathbb{R}_{>0})^{d}$ is such that $v$ and $x^{(i)}(T_1)-x^{(j)}(T_1)$ are not orthogonal for all $i\neq j$.
Then 
\begin{equation*}
   \frac{\, \mathrm{d}}{\, \mathrm{d} t}\left(x_2^{(i)}-x_2^{(j)}\right)(t)= \alpha_1\left\langle v, \left(x^{(i)}-x^{(j)}\right)(t)\right\rangle \hspace{1cm} \text{ for } t\in[T_1, T_2],
\end{equation*}
hence $x_2^{(i)}(t)-x_2^{(j)}(t)$ has a sign whenever $t-T_1\ll1$, for all $i\neq j$.
Taking $\alpha_1$ sufficiently small and rescaling time appropriately, we deduce $x_2^{(i)}(T_2)\neq x_2^{(j)}(T_2)$ for all $i\neq j$, or equivalently
\begin{equation} \label{eq:cond1step1} 
\Phi^{T_2}_{\theta_2}\left(\Phi^{T_1}_{\theta_1}\left(x^{(i)}\right)\right)_2\neq \Phi^{T_2}_{\theta_2}\left(\Phi^{T_1}_{\theta_1}\left(x^{(j)}\right)\right)_2,\hspace{1cm}\text{ for all }i\neq j.
\end{equation}
Consider the backward equation
\begin{equation*}
\begin{cases}
\dot y(t)=\overline{w}(t)(y(t))_++\overline{b}(t)  &t\in[T_{4n+1},T],\\
y(T)=y_0,
\end{cases}
\end{equation*}
and note that $y$ also solves the forward equation \eqref{eq: node} for $\theta(t)=-\overline{\theta}(T-t)$.

Take $\theta_{4n+3}=(w_{4n+3},b_{4n+3})$ with
\begin{equation*} 
w_{4n+3}=0, \hspace{1cm} b_{4n+3}=-\beta_2 {\bf1}
\end{equation*}
for $\beta_2>0$ large enough to ensure
\begin{equation*}
    y_k^{(i)}(T_{4n+2})>0 \hspace{1cm} \text{ for all } (k, i)\in[d]\times[n].
\end{equation*}
Take $\theta_{4n+2}=(w_{4n+2},b_{4n+2})$ with
\begin{equation*} 
w_{4n+2}=-\alpha_2\begin{bmatrix} u_1 & \ldots & u_k & \ldots & u_d  \\ 0 & \ldots & 0 & \ldots & 0\\ \vdots & \ddots & \vdots & \ddots & \vdots \\ 0 & \ldots & 0 & \ldots & 0
\end{bmatrix}, \hspace{1cm} b_{4n+2}=0
\end{equation*}
where $\alpha_2>0$ and $u=(u_1,\ldots,u_d)\in(\mathbb{R}_{>0})^d$ is such that $u$ and $y^{(i)}-y^{(j)}$ are not orthogonal for all $i\neq j$. Setting
\begin{equation*}
    \overline{y}^{(i)} \coloneqq \Phi^{T_{4n+1}}_{-\theta_{4n+2}}\left(\Phi^{T_{4n+2}}_{-\theta_{4n+3}}\left(y^{(i)}\right)\right)
\end{equation*}
for $i\in[n]$, arguing just as above, by choosing $\alpha_2$ sufficiently small, it ensues
\begin{equation*}  
\overline{y}^{(i)}_1\neq \overline{y}^{(j)}_1 \hspace{1cm} \text{ for all } i\neq j.
\end{equation*}

\subsubsection*{Step 2.}

Due to \eqref{eq:cond1step1}, we can relabel all points as to have 
\begin{equation} \label{eq: order.second}
    x_2^{(i)}(T_2)<x_2^{(i+1)}(T_2) \hspace{1cm} \text{ for all } i\in[n-1].
\end{equation}
We now show that 
\begin{equation}\label{eq:condstep2}
    x_1^{(i)}(T_{2n+1}) = \overline{y}_1^{(i)} + c_1
\end{equation}
for some $c_1\in\mathbb{R}$ and all $i\in[n]$. We argue by induction. 
The base case $i=1$ trivially holds with $\theta_3 \equiv 0 $ and $c_1=x_1^{(1)}-\overline{y}_1^{(1)}$, at time $T_3$. Of course, since the points don't move, the order \eqref{eq: order.second} is propagated up to time $T_3$. Assume \eqref{eq: order.second} and \eqref{eq:condstep2} hold at time $T_{2n-1}$ for all $i\in[n-1]$. 
Take $\theta_{2n}=(w_{2n},b_{2n})$ with 
\begin{equation*}
    w_{2n}=0,\hspace{1cm} b_{2n}=-\alpha_3\operatorname{sign}\left(x_2^{(n-1)}(T_{2n-1})\right)e_2,
\end{equation*}
where $\alpha_3>0$. Then for all $i\in[n]$, 
\begin{equation*}
\begin{aligned}
x_2^{(i)} (t) &= x_2^{(i)}(T_{2n-1}) - \alpha_3 t\operatorname{sign}\left(x_2^{(n-1)}(T_{2n-1})\right), \\
x^{(i)}_k (t) &= x^{(i)}_k (T_{2n-1}) \hspace{3cm} \text{ for } k \neq 2,
\end{aligned}
\quad \text{for } t \in [T_{2n-1}, T_{2n}].
\end{equation*}
By virtue of the order \eqref{eq: order.second} which holds at time $T_{2n-1}$ by heredity, we can choose $\alpha_3$ as to ensure 
\begin{equation}\label{eq:cond2step2}
    x_2^{(1)}(T_{2n})<\ldots < x_2^{(n-1)}(T_{2n}) < 0 < x_2^{(n)}(T_{2n}).
\end{equation}
All the while,
\begin{equation*}
    x_k^{(i)} (T_{2n}) = x_k^{(i)} (T_{2n-1})
\end{equation*}
for all $k\neq 2$ and $i\in[n]$. Now take $\theta_{2n+1}=(w_{2n+1},b_{2n+1})$ with 
\begin{equation*}w_{2n+1}=\frac{4n+3}{Tx_2^{(n)}(T_{2n})}\begin{bmatrix} 0 & \omega & 0 & \cdots & 0 \\
0 & 0 & 0 & \cdots & 0 \\
\vdots & \vdots & \vdots & \ddots & \vdots  \\
0 & 0 & 0 & \cdots & 0
\end{bmatrix}, \hspace{1cm}  b_{2n+1}=0,
\end{equation*}
where \begin{equation*}
    \omega =\overline{y}_1^{(n)}-x_1^{(n)}(T_{2n})+c_1.
\end{equation*}
Because of \eqref{eq:cond2step2} we see that $x_1^{(n)}(T_{2n+1}) = \overline{y}_1^{(n)} + c_1$ holds, and we moreover have $x^{(i)}_1(T_{2n+1}) = x_1^{(i)}(T_{2n}) = \overline{y}_1^{(i)} + c_1$ for $i\in[n-1]$ because of the heredity assumption. All in all, we conclude that \eqref{eq:condstep2} holds.

\subsubsection*{Step 3.}
We will now apply this argument simultaneously to all coordinates. 
We relabel all points anew as to have 
\begin{equation} \label{eq: order.first}
    x_1^{(i)}(T_{2n+1})<x_1^{(i+1)}(T_{2n+1})
\end{equation}
for all $i\in[n-1]$. We show that 
\begin{equation}\label{eq:condstep3}
    x^{(i)}(T_{4n}) = \overline{y}^{(i)} + c
\end{equation}
for some $c\in\mathbb{R}^d$ and all $i\in[n]$. We again argue by induction. 
The base case trivially holds at time $T_{2n+2}$ by taking $\theta_{2n+2} \equiv 0 $ and $c_k=x_k^{(1)}(T_{2n+1})-\overline y_k^{(1)}$ for all $k\in[d]$. Again, since the points don't move, the order \eqref{eq: order.first} is propagated up to time $T_{2n+2}$. Assume that \eqref{eq: order.first} and \eqref{eq:condstep3} hold at time $T_{4n-2}$ for all $i\in[n-1]$. 
Take $\theta_{4n-1}=(w_{4n-1},b_{4n-1})$ with 
\begin{equation*}
    w_{4n-1}=0,\hspace{1cm} b_{4n-1}=-\alpha_4\operatorname{sign}\left(x_1^{(n-1)}(T_{4n-2})\right)e_1
\end{equation*}
for $\alpha_4>0$.
Then for all $i\in[n]$,
\begin{equation*}
\begin{aligned}
x_1^{(i)} (t) &= x_1^{(i)}(T_{4n-2})- \alpha_4 t \operatorname{sign}\left(x_1^{(n-1)}(T_{4n-2})\right), \\
x^{(i)}_k (t) &= x^{(i)}_k (T_{4n-2}), \hspace{2.8cm} \text{ for } k \neq 1,
\end{aligned}
\quad \text{for } t \in [T_{4n-2}, T_{4n-1}].
\end{equation*}
By virtue of \eqref{eq: order.first} which holds at time $T_{4n-2}$ by heredity, we can choose $\alpha_4$ as to ensure
\begin{equation}\label{eq:cond2step3}
  x_1^{(1)}(T_{4n-1}) < \ldots <  x_1^{(n-1)}(T_{4n-1}) < 0 < x_1^{(n)}(T_{4n-1}).
\end{equation}
All the while 
\begin{equation*}
    x_k^{(i)} (T_{4n-1}) = x_k^{(i)} (T_{4n-2})
\end{equation*}
for all $k\neq 1$ and $i\in[n]$. Now take $\theta_{4n}=(w_{4n},b_{4n})$ as 
\begin{equation*}w_{4n}=\frac{4n+3}{Tx_1^{(n)}(T_{4n-1})}\begin{bmatrix} 0 & 0 & \cdots & 0 \\
\omega_1 & 0 & \cdots & 0 \\
\vdots & \vdots & \ddots & \vdots  \\
\omega_{d-1} & 0 & \cdots & 0
\end{bmatrix}, \hspace{1cm}  b_{4n}=0,
\end{equation*}
where \begin{equation*}
    \omega_k = \overline{y}_k^{(n)}-x_k^{(n)}(T_{4n-1})+c_k
\end{equation*}
for $k\in[d-1]$. Because of \eqref{eq:cond2step3}, we see that 
$x^{(n)}(T_{4n}) = \overline{y}^{(n)} + c$, and we moreover have  $ x^{(i)}(T_{4n}) = x^{(i)}(T_{4n-1}) = \overline{y}^{(i)} + c$ for $i\in[n-1]$ because of the heredity assumption. All in all, we conclude that \eqref{eq:condstep3} holds. 

We can conclude the proof by taking $\theta_{4n+1}=(w_{4n+1},b_{4n+1})$ with
\begin{equation*}
w_{4n+1}=0,\hspace{1cm} b_{4n+1}=-c.\tag*{\qed}
\end{equation*}

\subsection{\texorpdfstring{Proof of \Cref{lemLTA}}{Proof of}}

Let $Y\in\mathscr{M}_{d\times n}(\mathbb{R})$ be the matrix whose $i$-th column, for $i\in[n]$, is $y^{(i)}\in\mathbb{R}^d$. We define $\gamma:[0,1]\times\mathscr{M}_{d\times n}(\mathbb{R})\to\mathscr{M}_{d\times n}(\mathbb{R})$ by
\begin{equation*}\gamma(s, X)=(1-s) X+ sY.
\end{equation*}
There exists $\varepsilon>0$ such for all $X\in\mathscr{M}_{d\times n}(\mathbb{R})$ with $\|Y-X\|_1\leq\varepsilon$, we have
\begin{equation}
\label{eq: condrange.2}
\operatorname{rank}(\upsigma(\gamma(s, X)))=n
\end{equation}
for all $s\in[0,1]$. Take any $X\in\mathscr{M}_{d\times n}(\mathbb{R})$ as above, and then consider the map 
$\mathscr{G}:[0,1]\times\mathscr{M}_{d\times d}(\mathbb{R})\to\mathscr{M}_{d\times n}(\mathbb{R})$ defined by
\begin{equation*}
\mathscr{G}(s, w)= w\upsigma(\gamma(s, X)),
\end{equation*}
with $\upsigma$ being applied element-wise to the matrix.
Observe that $\mathscr{G}(s,\cdot)$ is a linear map for any $s\in[0,1]$, and \eqref{eq: condrange.2} ensures that $\mathscr{G}(s, \cdot)$ is surjective. Therefore $\mathscr{G}(s,\cdot)$ has a right inverse $\mathscr{F}(s,\cdot):\mathscr{M}_{d\times n}(\mathbb{R})\to\mathscr{M}_{d\times d}(\mathbb{R})$, also a linear map, which takes the form 
\begin{equation*}
\mathscr{F}(s, v)= \operatorname*{argmin}_{w\in \operatorname{ker}(\mathscr{G}\left(s,\cdot)-v\right)}\|w\|_2.
\end{equation*}
The family of linear operators $(\mathscr{F}(s,\cdot))_{s\in[0,1]}$ is uniformly bounded in operator norm, that is to say, 
\begin{equation}\label{eq:unifbound}
\max_{s\in[0,1]}\|\mathscr{F}(s,\cdot)\|_{\mathscr{L}(\mathscr{M}_{d\times n}(\mathbb{R});\mathscr{M}_{d\times d}(\mathbb{R}))}\leq C
\end{equation}
for some constant $C>0$. For $t\in[0,T]$ set
\begin{equation*}
w(t)\coloneqq\mathscr{F}\left(\frac{t}{T},\frac{Y-X}{T}\right).
\end{equation*}
For $i\in[n]$, the $i$-th column $x^i(t)\in\mathbb{R}^d$ of 
\begin{equation*}
X(t)\coloneqq\gamma\left(\frac{t}{T}, X\right)=\left(1-\frac{t}{T}\right)X+\frac{t}{T}Y    
\end{equation*} 
satisfies
\begin{align*}
\begin{cases}
    \dot x^i(t)=w(t)\upsigma(x^i(t))\hspace{1cm}\text{for }t\in[0,T],\\
    x^i(0)=x^{(i)},\\
    x^i(T)=y^{(i)}.
\end{cases}
\end{align*}
By virtue of \eqref{eq:unifbound},
\begin{equation*}
\|w(t)\|_2=\left\|\mathscr{F}\left(\frac{t}{T},\frac{Y-X}{T}\right)\right\|_2\lesssim\frac{C}{T}\|Y-X\|_1,
\end{equation*}
as desired.\qed

\bibliographystyle{alpha}
\bibliography{biblio}

\newpage
    
	\bigskip

\begin{minipage}[t]{.5\textwidth}
  {\footnotesize{\bf Antonio \'Alvarez-L\'opez}\par
  Departamento de Matem\'aticas\par
  Universidad Aut\'onoma de Madrid\par
  C/ Francisco Tom\'as y Valiente 7, \par
  28049 Madrid, Spain\par
 \par
  e-mail: \href{mailto:blank}{\textcolor{blue}{\scriptsize antonio.alvarezl@uam.es}}
  }
\end{minipage}
\begin{minipage}[t]{.5\textwidth}
{\footnotesize{\bf Borjan Geshkovski}\par
  Inria \&
  Laboratoire Jacques-Louis Lions\par
  Sorbonne Université\par
  4 Place Jussieu\par
  75005 Paris, France\par
 \par
  e-mail: \href{mailto:borjan@mit.edu}{\textcolor{blue}{\scriptsize borjan.geshkovski@inria.fr}}
  }
\end{minipage}%

\begin{center}
\begin{minipage}[t]{.5\textwidth}
  {\footnotesize{\bf Domènec Ruiz-Balet}\par
  CEREMADE, UMR CNRS 7534\par
  Université Paris-Dauphine, Université PSL\par
  Pl. du Maréchal de Lattre de Tassigny\par
  75016 Paris, France\par
 \par
  e-mail: \href{mailto:blank}{\textcolor{blue}{\scriptsize domenec.ruiz-i-balet@dauphine.psl.eu}}
  }
\end{minipage}%
\end{center}

\end{document}